\documentclass[letterpaper, 10 pt, conference]{ieeeconf}  

\IEEEoverridecommandlockouts                              

\usepackage{amsmath}
\usepackage{amssymb}
\usepackage{amstext}

\usepackage{graphicx}

\usepackage{algorithm}
\usepackage{algorithmic}

\usepackage{color}

\usepackage{url}

\usepackage{psfrag}

\usepackage[normalem]{ulem}

\usepackage{subcaption}

\usepackage{epstopdf}

\usepackage{booktabs}

\usepackage{tikz}
\usetikzlibrary{arrows}
\usepackage{verbatim}

\usepackage{standalone}

\usepackage{bm}

\usepackage{wasysym}

\usepackage{euscript}

\usepackage{latexsym}

\newtheorem{thm}{Theorem}

\newtheorem{lemma}[thm]{Lemma}

\graphicspath{{Figures/}}

\title{\LARGE \bf
Multivariable analytic interpolation with complexity constraints:\\ A modified Riccati approach
}

\author{Yufang Cui$^{1}$,  and Anders Lindquist$^{2}$
\thanks{$^{1}$Department of Automation, Shanghai
Jiao Tong University, Shanghai, China. {\tt\small cui-yufang@sjtu.edu.cn}}%
\thanks{$^{2}$Department of Automation and School of Mathematical Sciences, Shanghai
Jiao Tong University, Shanghai, China. {\tt\small alq@math.kth.se}}%
}

\begin{document}

\maketitle
\thispagestyle{empty}
\pagestyle{empty}

\begin{abstract}
Analytic interpolation problems with rationality and derivative constraints occur in many applications in systems and control. In this paper we present a new method for the multivariable case, which generalizes our previous results on the scalar case. This turns out to be quite nontrivial, as it poses many new problems. A basic step in the procedure is to solve a Riccati type matrix equation. To this end, an algorithm based on homotopy continuation is provided.
\end{abstract}

\section{Introduction}\label{intro}

A common problem in robust control and spectral estimation is to find an $\ell\times \ell$ matrix-valued rational function $F$, analytic in the unit disc $\mathbb{D}=\{z\mid |z|>1\}$, such that 
\begin{equation}
\label{F+F*}
F(e^{i\theta}) + F(e^{-i\theta})' > 0, \quad -\pi\leq \theta\leq\pi ,
\end{equation}
which also satifies the interpolation condition 
\begin{align}
\label{interpolation}
 \frac{1}{k!}F^{(k)}(z_{j})=W_{jk},\quad& j=0,1,\cdots,m,   \\
    &   k=0,\cdots n_{j}-1 ,\notag
\end{align}
where $'$ denotes transposition, $F^{(k)}(z)$ is the $k$th derivative of $F(z)$, and $z_0,z_1,\dots,z_m$ are distinct points in $\mathbb{D}$. We restrict the complexity of the rational function $F(z)$ by requiring that its McMillan degree be at most $\ell n$, where
\begin{equation}
\label{deg(f)}
n=(\sum_{j=0}^{m}n_j -1) .
\end{equation}
Without loss of generality we may assume that $ z_{0}=0$ and $W_{0}=\frac{1}{2}I$.
Then $F(z)$  has a realization 
 \begin{equation}
\label{ }
F(z)=\tfrac12 I + zH(I-zF)^{-1}G,
\end{equation}
where $H\in\mathbb{R}^{\ell\times\ell n}$, $F\in\mathbb{R}^{\ell n\times\ell n}$, $G\in\mathbb{R}^{\ell n\times\ell}$, the matrix $F$ has all its eighenvalues in $\mathbb{D}$ and $(H,F)$ is an observable pair. 

Let $W$ be the $\ell(n+1)\times\ell(n+1)$ matrix
\begin{equation}
\label{W}
W:=\begin{bmatrix}
W_{0}&~&~\\
~&\ddots&~\\
~&~&W_{m}
\end{bmatrix}
\end{equation}
with
\begin{equation}\label{Wj}
W_{j}=\begin{bmatrix}
W_{j0}&~&~&~\\
W_{j1}&W_{j0}&~&~\\
\vdots&\ddots&\ddots&~\\
W_{jn_{j-1}}&\cdots&W_{j1}&W_{j0}
\end{bmatrix}
\end{equation}
for each $j=0,1,\dots,m$. Moreover, let $Z$ be the $(n+1)\times (n+1)$ matrix 
\begin{equation}\label{Z}
Z:=\begin{bmatrix}
Z_{0}&~&~\\
~&\ddots&~\\
~&~&Z_{m}
\end{bmatrix}
\end{equation}
with 
\begin{equation}
\label{Zj}
Z_{j}=\begin{bmatrix}
z_{j}&~&~&~\\
1&z_{j}&~&~\\
~&\ddots&\ddots&~\\
~&~&1&z_{j}
\end{bmatrix} \quad j=0,1,\dots,m.
\end{equation}
Finally define the $n+1$-dimensional column vector
\begin{equation}\label{e}
e:=[e_{n_{0}}^{1},e_{n_{1}}^{1},\cdots,e_{n_{m}}^{1}]',
\end{equation}
where $e_{n_{j}}^{1}=[1,0,\cdots,0]\in\mathbb{R}^{n_j}$ for each $j=0,1,\dots,m$, and let $S$ be the unique solution of the Lyapunov equation
\begin{equation}
\label{Lyapunov}
S=ZSZ^*+ee' .
\end{equation}
Note that the eigenvalues of $Z$ are all located in the open unit disc $\mathbb{D}$.

The problem of determining the interpolant $F(z)$ is an inverse problems which has a solution if and only if 
\begin{equation}
\label{Pickmatrix}
W(S\otimes I_\ell) + (S\otimes I_\ell)W^* > 0 
\end{equation} 
(see, e.g., \cite{BLN}), and then there are an infinite number of solutions. We would like to find a parametrization of these  solutions. 

The special case when $\ell=1$, $m=0$ and $n_0=n+1$ is called the {\em rational covariance extension problem\/}  and was first formulated by Kalman \cite{Kalman-81} and then solved in steps in \cite{Gthesis,G87,BLGuM,BLpartial}, where a complete parameterization in terms of spectral zeros was obtained, and in \cite{BGuL,SIGEST}, where a convex optimization approach was introduced. This problem have occurred in many applications in systems and control such as  in signal and speech processing \cite{DelsarteGeninKamp} and in identification \cite{LPbook}. 
The case $n_0=n_1=\dots =n_m=1$ and $m=n$ is called the {\em Nevanlinna-Pick interpolation problem with degree constraint} and was early considered in robust control \cite{DFT} and many other applications in systems and control \cite{YS,DGK81}. It was completely  parameterized, again in steps, in \cite{GeorgiouNP,geo1999,BGL2001,BGL2000}, and a convex optimization approach was introduced in \cite{BGL2001,BGL2000}. Since then a large number of papers on the more general scalar problem has appeared \cite{BN,GL1,BLkimura,PF2006,FPRpicci}. We refer to \cite{LPbook} for further references.

The multivariable case $(\ell>1)$ is much harder, and the nice spectral-zero assignability present in the scalar case appears to be lost or at lease elusive. Restrictive classes of such problems have been considered in large number of papers \cite{Gthesis,BLN,G2006,G2007,FPZbyrneslindquist,RFP,Avventi,Takyar,ZhuBaggio,Zhu}, but the theory remains incomplete, and many problems have been left open.

In \cite{CLccdc19} we presented a complete parameterization of the problem presented above for the scalar case ($\ell =1$) in terms of a modified Riccati equation, which was first introduced for  more restricted classes of interpolation problems in \cite{BLpartial} and \cite{Lascc}. As \cite{BLpartial} studied the rational covariance extension problem, the modified Riccati equation was named the {\em Covariance Extension Equation} (CCE), and we retain this name although the problems now considered are much more general. 

In the present paper we take a first step in generalizing the results in \cite{CLccdc19} to the multivariable case $(\ell>1)$. In Section~\ref{prel} we provide the basic tools for the multivariable problem.  To describe our ultimate goal we provide in Section~\ref{scalar} a brief review of the scalar results in \cite{CLccdc19}, and then in Section~\ref{sec:matrixcase} we develop the multivariable case in the same spirit. In Section~\ref{main} we present our main results and an algorithm based on homotopy continuation in the style of \cite{BFL,CLccdc19}.  The results fall somewhat short of what the scalar case promises, and, given some results in \cite{Takyar}, we suspect that this is due to problems introduced by the nontrivial Jordan structure of the multivariable case.  In Section\ref{modelred} we provide some simulations to illustrate this and also an example of model reduction.  Finally, in  Section~\ref{conclusion} we give some concluding remarks and suggestions for future research.

\section{Preliminaries}\label{prel}

Defining $\Phi_+(z):=F(z^{-1})$ we have
\begin{equation}
\label{ }
\Phi_+(z)=\tfrac12 I + H(zI-F)^{-1}G,
\end{equation}
which has all its poles in the unit disc $\mathbb{D}$. In view of  \eqref{F+F*} 
\begin{displaymath}
\Phi_+(e^{i\theta})+\Phi_+(e^{-i\theta})'>0, \quad -\pi\leq \theta\leq\pi ,
\end{displaymath}
and hence $\Phi_+(z)$ is (strictly) positive real \cite[Chapter 6]{LPbook}. By a coordinate transformation $(H,F,G)\to(HT^{-1},TFT^{-1},TG)$ we can choose $(H,F)$ in the observer canonical form
\begin{displaymath}
H=\text{diag}(h_{t_1},h_{t_2},\dots,h_{t_\ell}) \in \mathbb{R}^{\ell\times n\ell}
\end{displaymath}
with $h_\nu:=(1,0,\dots,0)\in\mathbb{R}^\nu$, and
\begin{equation}
\label{F}
F=J-AH \in\mathbb{R}^{n\ell\times n\ell}
\end{equation}
where $J:=\text{diag}(J_{t_1},J_{t_2},\dots, J_{t_\ell})$ with $J_\nu$ the $\nu\times\nu$ shift matrix
\begin{displaymath}
J_\nu =\begin{bmatrix}0&1&0&\dots&0\\0&0&1&\dots&0\\\vdots&\vdots&\vdots&\ddots&0\\
0&0&0&\dots&1\\0&0&0&\dots&0\end{bmatrix}
\end{displaymath}
and $A\in\mathbb{R}^{n\ell\times \ell}$.
The numbers $t_1,t_2,\dots,t_\ell$ are the {\em observability indices\/} of $\Phi_+(z)$, and 
\begin{equation}
\label{tsum}
t_1+t_2+\dots+t_\ell=n\ell.
\end{equation}
Next define $\Pi(z):=\text{diag}(\pi_{t_1}(z),\pi_{t_2}(z),\dots,\pi_{t_\ell}(z))$, where $\pi_\nu(z)=(z^{\nu-1},\dots,z,1)$, and the $\ell\times\ell$ matrix polynomial
\begin{equation}
\label{A(z)}
A(z)=D(z) +\Pi(z)A,
\end{equation}
where 
\begin{equation}
\label{D(z)}
D(z):=\text{diag}(z^{t_1},z^{t_2},\dots, z^{t_\ell}).
\end{equation}

\begin{lemma}\label{Ainvlem}
$H(zI-F)^{-1}=A(z)^{-1}\Pi(z)$
\end{lemma}

\begin{proof}
First note that $$\Pi(z)(zI-J)=\text{diag}(z^{t_1},z^{t_2},\dots, z^{t_\ell})H.$$ Then
\begin{displaymath}
\Pi(z)(zI-F)=\Pi(z)(zI-J)+\Pi(z)AH=A(z)H
\end{displaymath}
as claimed.
\end{proof}
Consequently
 \begin{equation}
\label{AinvB}
\Phi_+(z)=\tfrac12 A(z)^{-1}B(z),
\end{equation}
where
\begin{displaymath}
B(z)=D(z) +\Pi(z)B
\end{displaymath}
with
\begin{equation}
\label{AG2B}
B=A+2G.
\end{equation}

Moreover let $V(z)$ be the minumum-phase spectral factor of
\begin{displaymath}
V(z)V(z^{-1})'=\Phi(z) := \Phi_+(z) + \Phi_+(z^{-1})' .
\end{displaymath}
We know \cite[Chapter 6]{LPbook} that $V(z)$ has a realization of the form
\begin{displaymath}
V(z)=H(zI-F)^{-1}K + R,
\end{displaymath}
which, by Lemma~\ref{Ainvlem}, can be written
\begin{equation}
\label{ }
V(z)=A(z)^{-1}\Sigma(z)R,
\end{equation}
where
\begin{equation}
\label{Sigma(z)}
\Sigma(z)=D(z)+\Pi(z)\Sigma 
\end{equation}
with 
\begin{equation}
\label{Sigma}
\Sigma = A+KR^{-1}. 
\end{equation}

From stochastic realization theory \cite[Chapter 6]{LPbook} we have 
\begin{align}
  K  & =(G-FPH')(R')^{-1}  \label{K}\\
  RR'  &  = I-HPH' \label{R}
\end{align}
where $P$ is the unique minimum solution of the algebraic Riccati equation
\begin{equation}
\label{Riccati}
P=FPF' + (G-FPH')(I-HPH')^{-1}(G-FPH')'.
\end{equation}
Now, from \eqref{F}, \eqref{K} and \eqref{R} we have
\begin{align*}
   G &= JPH'  -AHPH' +KR^{-1}(I-HPH') \\
      & = \Gamma PH' +KR^{-1},
\end{align*}
where, in view of \eqref{Sigma},
\begin{equation}
\label{Gamma}
\Gamma=J-\Sigma H.
\end{equation}
Hence
\begin{equation}
\label{G}
G=\Gamma PH'+\Sigma -A.
\end{equation}
Since $F=\Gamma +KR^{-1}H$ and $G-\Gamma PH'=KR^{-1}$, \eqref{Riccati} can be written
\begin{align*}
   P & = (\Gamma +KR^{-1}H)P(\Gamma +KR^{-1}H)' +KK'\\
    &  = \Gamma P\Gamma' +\Gamma PH'(KR^{-1})' +KR^{-1}HP\Gamma' \\
    & \phantom{xxxxxxxxxxxxxxxxxxxx} +KR^{-1}(KR^{-1})' ,
\end{align*}
where we have also used \eqref{R}. Inserting $KR^{-1}=G-\Gamma PH'$ we have
\begin{equation}
\label{AREmod}
P=\Gamma (P-PH'HP)\Gamma' +GG' .
\end{equation}

\section{A review of the scalar case}\label{scalar}

To motivate our approach to the multivariable problem presented in Section~\ref{intro} we shall briefly review some  results on the scalar case ($\ell=1$) presented in \cite{CLccdc19}. To stress the fact that the matrices $H,G,A,B$ and $\Sigma$ are now $n$-vectors we shall here denote them $h,g,a,b$ and $\sigma$ instead, and the scalar $R$ will be denoted $\rho$. 

Introducing the  interpolation conditions \eqref{interpolation} into the calculation we obtain
\begin{displaymath}
g=u+U\sigma+U\Gamma Ph,
\end{displaymath}
where $P$ is the unique solution of the algebraic Riccati equation \eqref{AREmod}. Elimination $g$ we then have the modified Riccati equation 
\begin{equation}\label{scalarCCE}
\begin{split}
P=&\Gamma(P-Phh'P)\Gamma'\\
&+(u+U\sigma+U\Gamma Ph)(u+U\sigma+U\Gamma Ph)' ,
\end{split}
\end{equation}
where $u\in\mathbb{R}^n$ and $U\in\mathbb{R}^{n\times n}$ are given by 
\begin{equation}
\label{ }
\begin{bmatrix}u&U\end{bmatrix}:=\begin{bmatrix}0&I_{n}\end{bmatrix} M
\end{equation}
with
\begin{subequations}
\begin{equation}
\label{M}
M=\begin{bmatrix}e&V\end{bmatrix}^{-1}(W+\frac{1}{2}I)^{-1}(W-\frac{1}{2}I)\begin{bmatrix}e&V\end{bmatrix}
\end{equation}
and
\begin{equation}
\label{Vscalar}
V:=\begin{bmatrix}Ze&Z^{2}e&\cdots&Z^n e\end{bmatrix} .
\end{equation}
\end{subequations}
We showed in \cite{CLccdc19} that there is a map sending $W$ to $u$ which is a diffeomorphism and that there is a linear map $L$ such that $U=Lu$. 

Let $\mathcal{S}_{n}$ be the space of Schur polynomials (i.e., polynomials with all zeros in the open unit disc $\mathbb{D}$) of the form 
\begin{equation}\label{a}
a(z)=z^{n}+a_{1}z^{n-1}+\cdots+a_{n} ,
\end{equation}
and let $\mathcal{P}_n$ be the $2n$-dimensional  space of pairs $(a,b)\in\mathcal{S}_{n}\times\mathcal{S}_{n}$ such that $b(z)/a(z)$ is positive real. Moreover, for each $\sigma\in\mathcal{S}_{n}$, let $\mathcal{P}_n(\sigma)$ be the submanifold of $\mathcal{P}_n$ for which 
\begin{equation}
\label{ab2sigma}
a(z)b(z^{-1})+b(z)a(z^{-1})=2\rho^{2}\sigma(z)\sigma(z^{-1})
\end{equation}
holds, where $\rho^2$ is the appropriate normalizing factor. It was shown in \cite{BLduality} that  $\{\mathcal{P}_n(\sigma)\mid \sigma\in\mathcal{S}_{n}\}$ is a {\em foliation\/} of $\mathcal{P}_n$, i.e., a family of smooth nonintersecting submanifolds, called {\em leaves}, which together cover  $\mathcal{P}_n$. Moreover, for any polynomial \eqref{a}, let $a_*(z)=z^n a(z^{-1})$ be the reversed polynomial of $a(z)$. Finally, let $\mathcal{W}_+$ be the space of all $W$ such that the generalized Pick matrix $WS+SW^*$ is positive definite, where $S$ the unique solution of the Lyapunov equation \eqref{Lyapunov} .

The following result was proved in \cite{CLccdc19}.

\medskip

\begin{thm}
Let $\ell=1$. For each $(\sigma,W)\in\mathcal{S}_n\times\mathcal{W}_+$, the modified Riccati equation \eqref{scalarCCE}  has a unique  positive definite solution $P$ such that $hPh'<1$, and the problem to find a rational function $b_*(z)/a_*(z)$ satisfying the interpolation conditions \eqref{interpolation} and the positivity condition \eqref{ab2sigma} has a unique solution given by
\begin{equation}\label{P2ab}
\begin{split}
a&=(I-U)(\Gamma Ph+\sigma)-u\\
b&=(I+U)(\Gamma Ph+\sigma)+u
\end{split}
\end{equation}
In fact, the map sending $(a,b)\in\mathcal{P}_n(\sigma)$ to $W\in\mathcal{W}_+$ is a diffeomorphism. Finally, the degree of the interpolant equals the rank of $P$.
\end{thm}

\medskip

Consequently, by \eqref{ab2sigma}, for each $\sigma\in\mathcal{S}_n$ there is a unique interpolant $b_*(z)/a_*(z)$ with the prescribed properties such that
\begin{displaymath}
\rho^2\frac{\sigma(z)\sigma(z^{-1})}{a(z)a(z^{-1})}=\frac12\left[\frac{b(z)}{a(z)} +\frac{b(z^{-1})}{a(z^{-1})} \right] .
\end{displaymath}
Hence 
\begin{displaymath}
V(z)= \rho\frac{\sigma(z)}{a(z)}
\end{displaymath}
is the correspondning spectral factor.

In \cite{CLccdc19} we solved \eqref{scalarCCE} by homotopy continuation by taking $u(\lambda)=\lambda u$ with $\lambda$ varying from $0$ to $1$. We showed that this provides an efficient and robust algorithm for analytic interpolation with degree constraint that can handle situations which are difficult with the optimization approach, especially when system poles are close to the unit circle. 

\section{The matrix case}\label{sec:matrixcase}

Next we turn to the general multivariable case and introduce the interpolation condition \eqref{interpolation} in the matrix setting of Section~\ref{prel}. 

\begin{lemma}
Let the matrices $W$ and $Z$ be given by \eqref{W} and \eqref{Z}, respectively. Then the interpolation condition \eqref{interpolation} can be written
\begin{equation}
\label{interpolation2}
F(Z\otimes I_\ell)=W ,
\end{equation}
where $\otimes$ denotes Kronecker product.
\end{lemma}

\begin{proof}
Since $F(z)$ is analytic in $\mathbb{D}$, it has the representation 
\begin{displaymath}
F(z)=\sum_{k=0}^\infty C_k z^k
\end{displaymath}
for all $z\in\mathbb{D}$, where $C_0=\frac12 I_\ell$. A straight-forward but tedious calculation, omitted here for lack of space,
yields
\begin{displaymath}
F(Z_j\otimes I_\ell)= \sum_{k=0}^\infty (Z_j)^k\otimes C_k =W_j,
\end{displaymath}
where $W_j$, defined by \eqref{Wj}, is given by \eqref{interpolation}. 
Then \eqref{interpolation2} follows from \eqref{Z} and \eqref{W}.
\end{proof}

\medskip

Let $A_*(z)$ be the reversed matrix polynomial 
\begin{equation}
\label{A*}
A_*(z)=D(z)A(z^{-1})=I_\ell +D(z)\Pi(z^{-1})A,
\end{equation}
where $D(z)$ is given by \eqref{D(z)}, and let $B_*(z)$ be defined in the same way in terms of $B(z)$. Then
\begin{equation}
\label{AB2F}
F(z)=\tfrac{1}{2}A_*(z)^{-1}B_*(z)
\end{equation}
and the interpolation condition  \eqref{interpolation2} can be written
\begin{equation}
\label{interpolation3}
2A_*(Z\otimes I_\ell)W=B_*(Z\otimes I_\ell).
\end{equation}
Moreover, let the $\ell\times\ell n$ matrices $N_1, N_2, \dots, N_t$ be defined by
\begin{equation}
\label{N}
D(z)\Pi(z^{-1})=N_1z+N_2z^2+\dots +N_tz^t,
\end{equation}
where $t$ is the largest observability index. Then 
\begin{displaymath}
A_*(z) =I_\ell + A_1z+A_2z^2+\dots+A_tz^t,
\end{displaymath}
where $A_k=N_kA$. For later use we observe that 
\begin{equation}\label{Nmatrix}
N=\begin{bmatrix}N_{1}\\\vdots\\N_{t}\end{bmatrix}\in\mathbb{R}^{\ell t\times \ell n}, 
\; N_{k}=\begin{bmatrix}e_{t_1}^{k}\\~&e_{t_2}^{k}\\~&~&\ddots\\~&~&~&e_{t_\ell}^{k}\end{bmatrix}
	\end{equation}
where $e_{j}^{k} $ is a $1\times j$ row vector with the k:th element 1 and the others 0 whenever $k\leq j$, and a zero row vector of dimension $1\times j$ when $k>j$.

Next we reformulate \eqref{interpolation3} as
\begin{equation}
\label{interpolation4}
M\begin{bmatrix}
I_{\ell(n+1)}\\
I_{n+1}\otimes A_{1}\\
\vdots\\
I_{n+1}\otimes A_{t}
\end{bmatrix}W=\frac12M\begin{bmatrix}
I_{\ell(n+1)}\\
I_{n+1}\otimes B_{1}\\
\vdots\\
I_{n+1}\otimes B_{t}
\end{bmatrix} ,
\end{equation}
where $M$ is the $\ell(n+1)\times\ell(n+1)(t+1)$ matrix
\begin{displaymath}
M= \begin{bmatrix} (I_{\ell(n+1)}& Z\otimes I_\ell &(Z\otimes I_\ell)^2& \cdots & (Z\otimes I_\ell)^t\end{bmatrix}.
\end{displaymath}
In view of \eqref{AG2B}, \eqref{interpolation4} can be written
\begin{displaymath}
M\begin{bmatrix}
I_{\ell(n+1)}\\
I_{n+1}\otimes A_{1}\\
\vdots\\
I_{n+1}\otimes A_{t}
\end{bmatrix}(W-\frac{1}{2}I)=M\begin{bmatrix}
0_{\ell(n+1)}\\
I_{n+1}\otimes G_{1}\\
\vdots\\
I_{n+1}\otimes G_{t}
\end{bmatrix}
\end{displaymath}
or, equivalently,

\begin{equation}
\label{interpolation5}
M\begin{bmatrix}
0_{\ell(n+1)}\\
I_{n+1}\otimes G_{1}\\
\vdots\\
I_{n+1}\otimes G_{t}
\end{bmatrix}=
M\begin{bmatrix}
I_{\ell(n+1)}\\
I_{n+1}\otimes Q_1\\
\vdots\\
I_{n+1}\otimes Q_t
\end{bmatrix}T,
\end{equation}
where $Q_k:=A_k+G_k$, $k=1,2,\dots,t$, and 
\begin{equation}
\label{T}
T:=(W-\frac{1}{2}I)(W+\frac{1}{2}I)^{-1}
=\begin{bmatrix}
T_{0}&~&~\\
~&\ddots&~\\
~&~&T_{m}
\end{bmatrix},
\end{equation}
where
\begin{equation}\label{Tj}
T_{j}=\begin{bmatrix}
T_{j0}&~&~&~\\
T_{j1}&T_{j0}&~&~\\
\vdots&\ddots&\ddots&~\\
T_{jn_{j-1}}&\cdots&T_{j1}&T_{j0}
\end{bmatrix}
\end{equation}
for $j=0,1,\dots,m $.

Using the rule $(A\otimes B)(C\otimes D)=(AC)\oplus (BD)$, valid for arbitrary matrices of appropriate dimensions, \eqref{interpolation5} takes the form
\begin{align*}
    &Z\otimes  G_1 +Z^2\otimes G_2 +  \dots + Z^t\otimes G_t   \\
    &= (I_{\ell(n+1)}+ Z\otimes  Q_1 +Z^2\otimes Q_2 +  \dots + Z^t\otimes Q_t)T.
\end{align*}
Then multiplying both sides from the right by $(e\otimes I_{\ell})$ and observing that 
\begin{displaymath}
(Z^k\otimes G_k)(e\otimes I_\ell)=(Z^k e)\otimes  G_k= (Z^k e\otimes I_\ell)G_k,
\end{displaymath}
we obtain 
\begin{equation}
\label{interpolation6}
V \begin{bmatrix} G_1\\  \vdots \\G_t \end{bmatrix}=\hat{T}+(Z\otimes  Q_1 +Z^2\otimes Q_2 +  \dots + Z^t\otimes Q_t)\hat{T},
\end{equation}
where $V$ is the $\ell(n+1)\times\ell t$ matrix
\begin{equation}
\label{V}
V:=\begin{bmatrix}Ze\otimes I_{\ell}&\cdots&(Z^{t}e)\otimes I_{\ell}\end{bmatrix} 
\end{equation}
and $\hat{T}$ is the $\ell(n+1)\times\ell$ matrix
\begin{equation}
\label{That}
\hat{T} :=T(e\otimes I_{\ell}),
\end{equation}
that is 
\begin{equation}
\label{That2}
\hat{T}=\begin{bmatrix}
\hat{T}_{0}\\\hat{T}_{1}\\\vdots\\\hat{T}_{m}
\end{bmatrix},\quad\text{where}\;
\hat{T}_{j}=
\begin{bmatrix}
T_{j0}\\
T_{j1}\\
\vdots\\
T_{jn_{j}-1}
\end{bmatrix} .
\end{equation}
Therefore, since $G_k=N_kG$ and $Q_k=N_kQ$, we have 
\begin{equation}
\label{interpolation7}
VNG=\hat{T}+(Z\otimes  N_1Q + \dots + Z^t\otimes N_tQ)\hat{T},
\end{equation}
where $N$ is given by \eqref{Nmatrix}.
Now, $VN$ is an $\ell(n+1)\times\ell n$ matrix in which the top $\ell$ rows are zero, since $z_0=0$. Therefore it takes the form
\begin{equation}
\label{VN}
VN=\begin{bmatrix} 0_{\ell\times\ell n}\\L\end{bmatrix}.
\end{equation}
For the moment assuming that the square matrix $L$ is nonsingular -- we shall later see that this is not always true -- $VN$ has a psuedo-inverse $(VN)^\dagger$, and hence \eqref{interpolation7} yields
\begin{equation}
\label{interpolation8}
G =(VN)^\dagger\hat{T} +(VN)^\dagger(Z\otimes  N_1Q + \dots + Z^t\otimes N_tQ)\hat{T}
\end{equation}
Since, by definition, $Q=A+G$, \eqref{G} yields
\begin{equation}
\label{GammaSigma2G}
G=u + U(\Gamma PH'+\Sigma),
\end{equation}
where $u:=(VN)^\dagger\hat{T}$ and $U: \mathbb{R}^{\ell n\times\ell}\to\mathbb{R}^{\ell n\times\ell}$ is the linear operator 
\begin{displaymath}
Q\mapsto (VN)^\dagger(Z\otimes  N_1Q + \dots + Z^t\otimes N_tQ)\hat{T}.
\end{displaymath}

Inserting \eqref{GammaSigma2G} into \eqref{AREmod} we obtain the modified Riccati equation 
\begin{equation}
\begin{split}
\label{CEE}
P =&\Gamma (P-PH'HP)\Gamma' \\
&+(u + U\Sigma +U\Gamma PH')(u + U\Sigma +U\Gamma PH')' .
\end{split}
\end{equation}
It was first introduced in \cite{BLpartial} for the scalar case $\ell=1$ and for the special case of covariance extension. Therefore it has been called the Covariance Extension Equation (CEE).

\section{Main results}\label{main}

Next we generalize the results of Section~\ref{scalar} to the general multivariable problem, which is considerably more difficult. Therefore several key questions will be left unanswered at this time. Nevertheless the theory in its present (preliminary) form does give an workable algorithm for large classes of problems.

\subsection{Basic results}

Now redefine $\mathcal{S}_n$ to be the class of $\ell\times\ell$ matrix polynomials \eqref{A(z)} such that $\det A(z)$ has all its zeros in the open unit disc $\mathbb{D}$. Clearly $\mathcal{S}_n$ consists of subclasses with different Jordan structure $J$ defined via \eqref{F}. In each such subclass $D(z)$ and $\Pi(z)$ in \eqref{A(z)}, as well as $N_1,N_2,\dots,N_t$ in \eqref{N}, are the same. Let $\mathcal{W}_+$ we the values in \eqref{interpolation} that satisfy the geralized Pick condition \eqref{Pickmatrix}. 

\medskip

\begin{lemma}\label{VNlem}
Let the $\ell n\times \ell n$ matrix $L$ be defined by \eqref{VN}. Then $L$ is nonsingular if and only if all observability indices are the same, i.e., $t_1=t_2=\dots=t_\ell =n$.
\end{lemma}

\medskip

\begin{proof}
Let us order the observability indices as $t_1\geq t_2\geq \dots\geq t_\ell$ and set $t:=t_1$. Then, by \eqref{tsum}, $t\geq n$. Since $(Z,e)$ is a reachable pair, 
\begin{equation}
\label{rank=n}
\text{rank\,} \begin{bmatrix}Ze&Z^{2}e&\cdots&Z^t e\end{bmatrix} =n.
\end{equation}

First assume that $t=n$. Then, since $\text{rank}(A\otimes B)=\text{rank}(A)\text{rank}(B)$, 
\begin{displaymath}
V=\begin{bmatrix}Ze&Z^{2}e&\cdots&Z^n e\end{bmatrix}\otimes I_\ell \in\mathbb{C}^{\ell n\times \ell n}
\end{displaymath}
has rank $\ell n$, and $N\in\mathbb{R}^{n\ell\times n\ell}$ given by
\begin{equation}\label{Nn}
N=\begin{bmatrix}N_{1}\\\vdots\\N_{n}\end{bmatrix}, \quad N_{k}=\begin{bmatrix}
		e_{n}^{k}\\
		~&e_{n}^{k}\\
		~&~&\ddots\\
		~&~&~&e_{n}^{k}
		\end{bmatrix}
	\end{equation}
also has rank $\ell n$.  Consequently Sylverster's inequality,
\begin{displaymath}
\text{rank\,}V+\text{rank\,}N-\ell n\leq \text{rank\,} VN \leq \text{min\,}(\text{rank\,}V,\text{rank\,}N),
\end{displaymath}
(see, e.g., \cite[p.741]{LPbook}) implies that $VN$ has rank $\ell n$, and hence $L$ is nonsingular. 

Next assume that $t>n$. Then the first $t$ columns of $N$, now given by \eqref{Nmatrix} can be written $I_t \otimes (e^1_\ell)'$, so the first $t$ columns of $VN$ form the matrix
\begin{displaymath} 
\begin{split}
\left(\begin{bmatrix}Ze&Z^{2}e&\cdots&Z^t e\end{bmatrix}\otimes I_\ell\right)\left(I_t \otimes(e^1_\ell)'\right)\\
=\begin{bmatrix}Ze&Z^{2}e&\cdots&Z^t e\end{bmatrix}\otimes (e^1_\ell)',
\end{split}
\end{displaymath}
which in view of \eqref{rank=n} has rank $n<t$. Hence the columns of $VN$ are linearly dependent, and thus $L$ is singular. 
\end{proof}

\medskip

In the present matrix case, the relation \eqref{ab2sigma} reads
\begin{equation}
\label{AB2Sigma}
A(z)B(z^{-1})'+B(z)A(z^{-1})'=2\Sigma(z)RR'\Sigma(z^{-1})'.
\end{equation}
Let $\mathcal{P}_n$ be the space of pairs $(A,B)\in\mathcal{S}_{n}\times\mathcal{S}_{n}$ such that $A(z)^{-1}B(z)$ is positive real. 
Then the  problem at hand is to find, for each $\Sigma\in\mathcal{S}_{n}$, a pair $(A,B)\in\mathcal{P}_n$ such that  \eqref{AB2Sigma} and \eqref{interpolation} hold. 

\medskip

\begin{thm}\label{mainthm}
Given $(\Sigma,W)\in\mathcal{S}_n\times\mathcal{W}_+$, where $\Sigma(z)$ has all it observability indicies equal.  Then to any positive definite solution $P$ of the Covariance Extension Equation \eqref{CEE} such that $HPH'< I$,  there corresponds a unique analytic interpolant \eqref{AB2F}, where $A(z)$ and $B(z)$ have the same Jordan structure as $\Sigma(z)$,  the matrices $A$ and $B$ are given by 
\begin{equation}
\label{AB}
\begin{split}
A&=(I-U)(\Gamma PH'+\Sigma)-u \\
B&=(I+U)(\Gamma PH'+\Sigma)+u
\end{split}
\end{equation}
and $A(z)$ and $B(z)$ satisfy \eqref{AB2Sigma} with 
\begin{equation}
\label{P2R}
R=(I-HPH')^{\frac{1}{2}}.
\end{equation}
Finally, 
\begin{equation}
\label{deg=rank}    
\deg F(z) = \text{rank\,} P.
\end{equation}
\end{thm}

\medskip

\begin{proof}
The theorem follows from the derivation above.
For the details of the proof of \eqref{deg=rank} we refer to \cite{BLpartial}. 
\end{proof}
\medskip

Let us stress that these results are considerably weaker than the corresponding theorem for the scalar case reviewed in Section~\ref{scalar}. In fact, Theorem~\ref{mainthm} does not guarantee that there exists a unique solution to \eqref{CEE}. In fact, if there were two solutions to \eqref{CEE}, there would be two interpolants, a unique one for each solution $P$. Moreover, the condition on the observability indices  restricts the classes of Jordan structures that are feasible. 

However, Theorem~\ref{mainthm} can be combined with other partial results on existence and uniqueness. There are multivariable problems for which we already know that there is a unique solution to the interpolation problem, and then existence and uniquenss of a solution  to \eqref{CEE} will follow.  A case in point is when $\Sigma(z)=\sigma(z) I$, where $\sigma(z)$ is a stable scalar polynomial \cite{BLN,FPZbyrneslindquist}, in which case the the observability indices are all equal, as required in Theorem~\ref{mainthm}.  In this case the analytic interpolation problem will have a unique solution, and thus, tracing the calculations in Section~\ref{prel} backwards, so will \eqref{CEE}. The same is true when $\Sigma(z)=\sigma(z) C$, where $C$ is full rank \cite{FPZbyrneslindquist}.

On the other hand, in recent years there have been a number of results \cite{FPZbyrneslindquist,RFP,Takyar,ZhuBaggio,Zhu} on the question of existence and uniqueness of the multivariate analytic interpolation problem, mostly for the covariance extension problem ($m=0, n_0=n+1$), but there are so far only partial results and for special structures of the prior (in our case $\Sigma(z)$). Especially the question of uniqueness has proven elusive. Perhaps, as suggested in \cite{Takyar}, this is due to the Jordan structure, and this could be the reason for the condition on the observability indices required in Theorem~\ref{mainthm}. In any case, as long as our algorithm delivers a solution to the Covariance Extension Equation, we will have a solution to the analytic interpolation problem, unique or not. An advantage of our method is that  \eqref{deg=rank} can be used for model reduction, as will be illustrated in Section~\ref{modelred}.

\subsection{Solving CEE by homotopy continuation}

We shall provide an algorithm for solving \eqref{CEE} based on homotopy continuation. We assume from now on that $t:= t_1=t_2=\dots,t_\ell =n$. Whenever this algoritm delivers a solution $P$, the interpolant is obtained via \eqref{AB}.  

When $u=0$, $\hat{T}=0$, and hence $U=0$. Then the modified Riccati equation \eqref{CEE} becomes $P=\Gamma(P-PH'HP)\Gamma'$, which has the solution $P=0$. We would like to make a continuous deformation of $u$ to go from this trivial solution to the solution of \eqref{CEE}, so we choose $u(\lambda)=\lambda u$ with $\lambda\in [0,1]$. The corresponding deformation of $U$ is $\lambda U$, and  $T$ is deformed to $\lambda T$. The value matrix \eqref{W} will vary as
\begin{displaymath}
W(\lambda)=(I-\lambda T)^{-1} -\tfrac12 I,
\end{displaymath}
 and we need to ascertain that $W(\lambda)$ still satisfies \eqref{Pickmatrix}
 along the whole trajectory.

\medskip

\begin{lemma}
Suppose that $W\in\mathcal{W}_+$. Then $W(\lambda)\in\mathcal{W}_+$ for all $\lambda\in [0,1]$.
\end{lemma}

\begin{proof}
By \eqref{Pickmatrix} we want to show that 
\begin{equation}
\label{Wlambda}
W(\lambda)E+EW(\lambda)^*>0
\end{equation}
 for $E:=S\otimes I_\ell$. We have
\begin{displaymath}
\begin{split}
&W(\lambda)E+EW(\lambda)^* \\
&=\left((I -\lambda T)^{-1}-\tfrac12 I\right)E+E\left(I-\lambda T^*)^{-1}-\tfrac12 I\right)\\
&= (I -\lambda T)^{-1}(E-\lambda^2TET^*) (I -\lambda T^*)^{-1}
\end{split}
\end{displaymath}
which we know is positive definite for $\lambda=1$. However,
\begin{displaymath}
E-\lambda^2TET^*\geq E-TET^*,
\end{displaymath}
and and therefore \eqref{Wlambda} holds  {\em a fortiori}. 
\end{proof}

\medskip

Now, note that equation \eqref{AB2Sigma} can be written as
\begin{equation}
\label{SM}
\begin{split}
&S(A)M(B)+S(B)M(A)\\
&=2S(\Sigma)(I_{n+1}\otimes RR')M(\Sigma)
\end{split}
\end{equation}
where
\begin{equation*}
S(A)=\begin{bmatrix}
I&A_{1}&\cdots&A_n\\
~&I&\cdots&A_{n-1}\\
~&~&\ddots&\vdots\\
~&~&~&I
\end{bmatrix}\qquad M(A)=\begin{bmatrix}
I\\
A_{1}'\\
\vdots\\
A_{n}'
\end{bmatrix}.
\end{equation*}

\medskip

\begin{lemma}\label{ANBn}
Let $N_n$ be defined by \eqref{Nn}. Then
\begin{displaymath}
A_n+B_n=2\Sigma_nRR' ,
\end{displaymath}
where $A_n=N_nA$, $B_n=N_nB$ and $\Sigma_n=N_n\Sigma$.
\end{lemma}

\begin{proof}
From \eqref{AB} and \eqref{R}  we have 
\begin{displaymath}
A+B=2(\Gamma PH'+\Sigma)=2(JPH' +\Sigma RR').
\end{displaymath}
Since $e^n_nJ_n=0$ and hence $N_nJ=0$, the statement of the lemma follows.
\end{proof}

\medskip

Applying Lemma~\ref{ANBn} and deleting the zero row in \eqref{SM}, it can be reduced to $n\ell\times \ell$ equations
\begin{align*}
\begin{bmatrix}
I_{n\ell}&0_{n\ell\times \ell}
\end{bmatrix}&(S(A)M(B)+S(B)M(A))\\
&=2\begin{bmatrix}I_{n\ell}&0_{n\ell\times \ell}\end{bmatrix}S(\Sigma)(I_{n+1}\otimes RR')M(\Sigma)
\end{align*}

Therefore, introducing the  $n\ell\times \ell$ matrix
\begin{equation}
p=PH' ,
\end{equation}
we use the homotopy 
\begin{equation}
\begin{split}
\mathcal{H}(p,\lambda):=&\begin{bmatrix}
I_{n\ell}&0_{n\ell\times \ell}
\end{bmatrix}\big(S(A)M(B)+S(B)M(A)\\
&-2S(\Sigma)(I_{n+1}\otimes (I-Hp))M(\Sigma)\big)=0 ,
\end{split}
\end{equation}
where
\begin{equation}
\begin{split}
A&=A(p,\lambda):=\Gamma p+\Sigma-\lambda u-\lambda U(\Gamma p +\Sigma) \\
B&=B(p,\lambda):=\Gamma p+\Sigma+\lambda u+\lambda U(\Gamma p +\Sigma)\\
\end{split}
\end{equation}
depend on $(p,\lambda)$. 
Then the problem reduces to solving the differential equation
\begin{equation}
\begin{split}
\label{diffequ}
&\frac{d}{d\lambda}\text{vec}(p(\lambda))=\left[\frac{\partial \text{vec}(\mathcal{H}(p,\lambda))}{\partial \text{vec}(p)}\right]^{-1}\frac{\partial \text{vec}(\mathcal{H}(p,\lambda))}{\partial\lambda}\\
&\text{vec}(p(0))=0
\end{split}
\end{equation}
\cite{AllgowerGeorg}, which has the solution $ \hat{p}(\lambda) $ for $ 0\leq\lambda\leq1$. The solution of \eqref{CEE} is then obtained by finding the unique solution of the  Lyapunov equation
\begin{equation}
\begin{split}
P-\Gamma P\Gamma' =&-\Gamma p(1)p(1)'\Gamma' \\
&+(u + U(\Gamma p(1) +\Sigma))(u + U(\Gamma p(1) +\Sigma))'.
\end{split}
\end{equation}

\section{Some simple illustrative examples}
Next we provide a few simulations that illustrate the theory. 

\subsection{Example 1}
We consider a problem with  the interpolation constraints
\begin{displaymath}
\begin{split}
&F(0)=\frac{1}{2}\begin{bmatrix}1&0\\0&1\end{bmatrix}\\
& F(0.5)=\begin{bmatrix}1&0\\0&0.4\end{bmatrix}\quad 
 F^{(1)}(0.5)=\begin{bmatrix}2&0.1\\0&0.1\end{bmatrix},
\end{split}
\end{displaymath}
where $ n=2,\ell=2$. This yields
\begin{equation}
	Z=\begin{bmatrix}
	0&~&~\\
	~&0.5&~\\
	~&1&0.5\\
	\end{bmatrix}\quad e=\begin{bmatrix}
	1\\
	1\\
	0
	\end{bmatrix}
	\end{equation}
	
Taking $\Sigma(z)$ of the form
\begin{equation}
\label{Sigma(z)2}
\Sigma(z)=\begin{bmatrix}z^2&0\\0&z^2\end{bmatrix}+\Pi(z)\Sigma
\end{equation}
we have $t_1=t_2=2$, and thus
\begin{displaymath}
VN=\begin{bmatrix}
	0_{2\times 4}\\
	L
	\end{bmatrix},\quad L=\begin{bmatrix}
	0.5&0.25&0&0\\
	0&0&0.5&0.25\\
	1&1&0&0\\
	0&0&1&1
	\end{bmatrix},
\end{displaymath}
where clearly $L$ is nonsingular as required. 
Then choosing
\begin{equation}
\label{Sigmavalu}
\Sigma=\begin{bmatrix}
1& 0.3\\
0.2& 0.3\\
0.1& 0.4\\
0.7& 0.2
\end{bmatrix},
\end{equation}
the matrix polynomial \eqref{Sigma(z)2} is stable, and we can use our algorithm to obtain
\begin{equation}
\label{A(z)B(z)}
\begin{split}
&A(z)=\begin{bmatrix}z^2&0\\0&z^2\end{bmatrix}+\Pi(z)A\\
&B(z)=\begin{bmatrix}z^2&0\\0&z^2\end{bmatrix}+\Pi(z)B ,
\end{split}
\end{equation}
where
\begin{displaymath}
A=\begin{bmatrix}
0.9467 &  -0.1737\\
0.3603 &   0.3583\\
-0.0445&    1.0925\\
0.2147  &  0.7364
  \end{bmatrix}\quad  B=\begin{bmatrix}
 -0.0533&    0.2263\\
-0.3517  & -0.2893\\
-0.2445  &  0.0925\\
0.2406   &-0.9739
  \end{bmatrix}
\end{displaymath}

If instead we choose $\Sigma(z)$ of the form
\begin{displaymath}
\Sigma(z)=\begin{bmatrix}z^3&0\\0&z\end{bmatrix}+\Pi(z)\Sigma,
\end{displaymath}
then $t=t_{1}=3,t_{2}=1$, and thus
  \begin{displaymath}
VN=\begin{bmatrix}
  0_{2\times 4}\\
  L
  \end{bmatrix}, \quad L=\begin{bmatrix}
0.5&0.25&0.125&0\\
  0&0&0&0.5\\
  1&1&0.75&0\\
  0&0&0&1
  \end{bmatrix},
\end{displaymath}
where $L$ is singular as anticipated by Lemma~\ref{VNlem}. Hence the algorithm cannot be used. 

\subsection{Example 2}
Next consider the multivariable covariance extension problem
\begin{displaymath}
F(0)=\tfrac12 I_2, \quad F^{(1)}(0)=C_1, \quad F^{(2)}(0)=C_2,
\end{displaymath}
where
\begin{displaymath}
C_1=\begin{bmatrix}
-0.5&0.2\\
-0.1&-0.5
\end{bmatrix}, \quad
C_2= 2\begin{bmatrix}
0.1&-0.6\\
0.1&-0.3
\end{bmatrix}.
\end{displaymath}
In this case 
\begin{displaymath}
Z=\begin{bmatrix}
  0&~&~\\
  1&0&~\\
  ~&1&0\\
  \end{bmatrix}\quad e=\begin{bmatrix}
  1\\
  0\\
  0
  \end{bmatrix}
\end{displaymath}
By Lemma~\ref{VNlem}, we need to choose $t_1=t_2=2$ and this yields 
\begin{displaymath}
VN=\begin{bmatrix}
  0_{2\times 4}\\
  L
  \end{bmatrix},\quad L=\begin{bmatrix}
1&     0&     0 &    0\\
0 &    0 &    1  &   0\\
0  &   1  &   0   &  0\\
0   &  0   &  0    & 1
  \end{bmatrix},
\end{displaymath}
where $L$ is nonsingular. Choosing the $\Sigma(z)$ in \eqref{Sigmavalu},
  our algoritm delivers the solution \eqref{A(z)B(z)} with
  \begin{displaymath}
A=\begin{bmatrix}
0.9467 &  -0.1737\\
0.3603 &   0.3583\\
-0.0445&    1.0925\\
0.2147  &  0.7364
  \end{bmatrix}\quad  B=\begin{bmatrix}
 -0.0533&    0.2263\\
-0.3517  & -0.2893\\
-0.2445  &  0.0925\\
0.2406   &-0.9739
  \end{bmatrix} .
\end{displaymath}
Again the choice  $t=t_{1}=3,t_{2}=1$ of obsevability indices does not work (Lemma~\ref{VNlem}).

\subsection{Model reduction}\label{modelred}
Consider a system with a $2\times 2$ transfer function 
\begin{equation}\label{truesystem}
V(z)=A(z)^{-1}\Sigma(z)
\end{equation}
of dimension $10$ and with observability indices $t_1=t_2=5$,
where $A(z)$ and $\Sigma(z)$ are given by  \eqref{A(z)} repectively \eqref{Sigma(z)} with
\begin{equation*}
A=\begin{bmatrix}
-0.11&   -0.02\\
-0.08 &  -0.15\\
0.05  &  0.10\\
-0.05 &  -0.09\\
-0.13 &  -0.09\\
0.11 &   0.07\\
0.09  &  0.19\\
-0.03 &  -0.03\\
-0.10 &  -0.13\\
0.12  &  0.05
\end{bmatrix}, \quad \Sigma=\begin{bmatrix}
0.1500 & 0\\
-0.6900 &  0\\
0.1025 &0\\
0.0306 &0\\
-0.0034 & 0\\
0  &  0.1500\\
0  & -0.6900\\
0  &  0.1025\\
0  &  0.0306\\
0 &  -0.0034\\
\end{bmatrix}.
\end{equation*}
Here $ \Sigma(z)=\sigma(z) I_2$, were 
\begin{displaymath}
\sigma(z)=(z-0.1)(z-0.3)(z-0.6)(z+0.2)(z+0.95).
\end{displaymath}
We pass (normalized) white noise though the system
\begin{figure}[thb!]
	\centering
	\includegraphics[width = 0.6\linewidth]{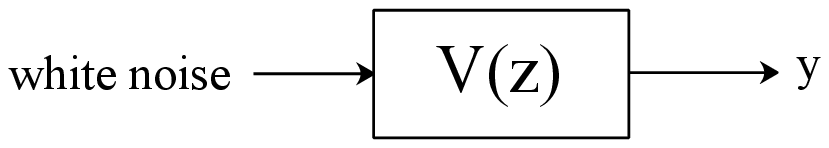}
\end{figure}\\
to obtain the output $y_{0},y_{2},\cdots,y_{N}$, and from this output data we estimate the $2\times 2$ matrix valued covariance sequence 
\begin{equation}
\hat{C}_{k}=\frac{1}{N-k+1}\sum_{t=k}^{N}y_{t}y_{t-k}' .
\end{equation}
Then we solve the problem \eqref{interpolation} with $\ell=2$, $m=0$, $n_0=6$, and $W_{0k}=\hat{C}_{k}$ for $k=0,2,\dots,5$. The modified Riccati equation \eqref{CEE} has a solution $P$ with eigenvalues
\begin{equation*}
\begin{split}
1.5\times10^{-6},2.7\times10^{-5},0.0007,0.0041,
\\0.0104,0.0338,0.1993,0.3457,0.6535,0.7138.
\end{split}
\end{equation*}
The first four eigenvalues are very small, so we can reduce the degree of this system from ten to  six by choosing the first four covariance lags $\hat{C}_0,\hat{C}_1,\hat{C}_2,\hat{C}_3$  and removing zeros of $\Sigma(z)$ at $-0.2$ and $0.6$. The reduced-order system will have observability indices $t_1=t_2=3$. The singular values of the true system \eqref{truesystem} are shown the multivariable Bode-type plot in Fig.~\ref{figure2} together with those of the estimated systems of degree ten and six, respectively. As can be seen, the reduced-order system is a good approximation.
\begin{figure}[thb!]
	\centering
	\includegraphics[width = 0.8\linewidth]{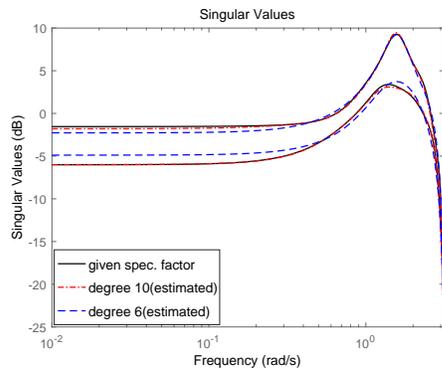}
	\caption{Estimated singular values and the true ones}
	\label{figure2}
\end{figure}

\section{Concluding remarks}\label{conclusion}

We have extended our previous results \cite{CLccdc19} for the scalar case to the matrix case. However, multivariable versions of analytic interpolation with rationality constraints have been marred by difficulties to establish existence and, in particular, uniqueness in the various parameterizations \cite{Gthesis,BLN,G2006,G2007,FPZbyrneslindquist,RFP,Takyar,ZhuBaggio,Zhu}, and we have encountered similar difficulties here. Our approach attacks these problems from a different angle and might put new light on these challenges.
Therefore future research efforts will be directed towards settling these intriguing open questions in the context of the modified Riccati equation \eqref{CEE}. 

%
%
%

\bibliographystyle{IEEEtran}

\end{document}